\documentclass[10pt]{article}
\usepackage{indentfirst, latexsym, amsmath,bm,amssymb}
\usepackage{tikz}
\usepackage{color,soul}
\usepackage{amsmath}
\usepackage{amssymb}
\usepackage{enumerate}
\usepackage{mathrsfs}
\usepackage{mathtools}
\usepackage{stmaryrd}
\usepackage{bbold}

\begin{document}

\newtheorem{theorem}{Theorem}[section]
\newtheorem{corollary}[theorem]{Corollary}
\newtheorem{definition}[theorem]{Definition}
\newtheorem{conjecture}[theorem]{Conjecture}
\newtheorem{question}[theorem]{Question}
\newtheorem{lemma}[theorem]{Lemma}
\newtheorem{remark}[theorem]{Remark}
\newtheorem{proposition}[theorem]{Proposition}
\newtheorem{example}[theorem]{Example}
\newenvironment{proof}{\noindent {\bf
Proof.}}{\rule{3mm}{3mm}\par\medskip}
\newcommand{\pp}{{\it p.}}
\newcommand{\de}{\em}

\newcommand{\JEC}{{\it Europ. J. Combinatorics},  }
\newcommand{\JCTB}{{\it J. Combin. Theory Ser. B.}, }
\newcommand{\JCT}{{\it J. Combin. Theory}, }
\newcommand{\JGT}{{\it J. Graph Theory}, }
\newcommand{\ComHung}{{\it Combinatorica}, }
\newcommand{\DM}{{\it Discrete Math.}, }
\newcommand{\ARS}{{\it Ars Combin.}, }
\newcommand{\SIAMDM}{{\it SIAM J. Discrete Math.}, }
\newcommand{\SIAMADM}{{\it SIAM J. Algebraic Discrete Methods}, }
\newcommand{\SIAMC}{{\it SIAM J. Comput.}, }
\newcommand{\ConAMS}{{\it Contemp. Math. AMS}, }
\newcommand{\TransAMS}{{\it Trans. Amer. Math. Soc.}, }
\newcommand{\AnDM}{{\it Ann. Discrete Math.}, }
\newcommand{\NBS}{{\it J. Res. Nat. Bur. Standards} {\rm B}, }
\newcommand{\ConNum}{{\it Congr. Numer.}, }
\newcommand{\CJM}{{\it Canad. J. Math.}, }
\newcommand{\JLMS}{{\it J. London Math. Soc.}, }
\newcommand{\PLMS}{{\it Proc. London Math. Soc.}, }
\newcommand{\PAMS}{{\it Proc. Amer. Math. Soc.}, }
\newcommand{\JCMCC}{{\it J. Combin. Math. Combin. Comput.}, }
\newcommand{\GC}{{\it Graphs Combin.}, }

\title{  {Lower bounds  for  the $\mathcal{A}_{\alpha}$-spectral radius of  uniform hypergraphs 
 }\thanks{
 This work is supported by the National Natural Science Foundation of China (Nos. 11971311, 
 12026230). 
E-mail addresses: zpengli@sjtu.edu.cn (P.-L. Zhang),  xiaodong@sjtu.edu.cn ($^\dag$X.-D. Zhang, corresponding author). }}
\date{}

\author{
 Peng-Li Zhang, Xiao-Dong Zhang$^\dag$\\
{\small  School of Mathematical Sciences,  MOE-LSC, SHL-MAC, Shanghai Jiao Tong University,}\\ {\small  Shanghai 200240, PR China} \\
}

\maketitle

\vspace{-0.5cm}

\begin{abstract}
For $0\leq \alpha < 1$,  the  $\mathcal{A}_{\alpha}$-spectral radius of   a $k$-uniform hypergraph $G$ is defined to be the spectral radius of the tensor  $\mathcal{A}_{\alpha}(G):=\alpha \mathcal{D}(G)+(1-\alpha)  \mathcal{A}(G)$, where  $\mathcal{D}(G)$ and $A(G)$    are diagonal and the adjacency tensors of $G$ respectively.  This paper presents several lower bounds for the difference between the $\mathcal{A}_{\alpha}$-spectral radius and an average degree $\frac{km}{n}$ for a connected $k$-uniform hypergraph with $n$ vertices and $m$ edges,
which may be considered as the measures of irregularity of $G$.
Moreover,  two lower bounds on the $\mathcal{A}_{\alpha}$-spectral radius  are obtained in terms of the  maximum  and minimum degrees of a hypergraph.

\end{abstract}



\section{Introduction}

Let $G$ be a hypergraph on $n$ vertices with vertex set $V(G)$ and edge set $E(G).$ The elements of $V=V(G),$ labeled as $\{v_{1},\ldots,v_{n}\},$ are referred to as vertices and the elements of $E=E(G)$ are called edges.
If $|e|=k$ for each $e \in E(G),$ then $G$ is said to be a $k$-uniform hypergraph. For $k=2,$ it refers the ordinary graph.
For a vertex $v_{i}\in V(G),$ we denote $E_{v_{i}}(G)=\{e\in E(G)|v_{i}\in e \},$ which is the set of edges containing the vertex $v_{i}.$
The degree $d_{G}(v_{i})$ (or simply $d_{v_{i}}$) of a vertex $v_{i} \in V(G)$ is defined as  $d_{v_{i}}=|e_{j}:v_{i} \in e_{j}\in E(G)|.$
Denote the maximum degree, the minimum degree and the average degree of $G$ by
$\Delta(G),\delta(G)$ and $\overline{d}(G),$ respectively.
A hypergraph $G$ is $d$-{\it regular} if $\Delta(G)=\delta(G)=d,$ otherwise, $G$ is irregular.
A {\it complete} $k$-uniform hypergraph is defined to be a hypergraph
$G=(V(G),E(G))$ with the edge set consisting of all $k$-subsets of $V(G).$ Obviously, a complete $k$-uniform hypergraph on $n$ vertices is $\binom{n-1}{k-1}$-regular. Here, we denote an $n$-vertex $k$-uniform complete hypergraph  by $K_{n}^{k}.$
The
{\it complement } of a $k$-uniform hypergraph $G$ is
the $k$-uniform hypergraph $\overline{G}$ with the  same vertex set as $G$ and the edge set of which consists of $k$-subsets of $V(G)$ not in $E(G).$
Moreover, for different $i,j\in V(G),$ $i$ and $j$ are said to be {\it adjacent},
written  $i\sim j,$  if there is an edge of $G$  containing both $i$ and $j.$ Two edges are said to be {\it adjacent} if their intersection is not empty. A vertex $v$ is said to be {\it incident} to an edge $e$ if $v \in e.$

A walk $W$ of length $\ell$ in $G$ is a sequence of alternate vertices and edges:  $ v_{1}e_{1}v_{2}e_{2}\ldots v_{\ell}e_{\ell}v_{\ell+1},$ where $\{v_{i},v_{i+1}\}\subseteq e_{i}$ for $i=1,\ldots,\ell.$
A {\it walk} of $G$ is called a {\it path} if no vertices or no edges are repeated.
A hypergraph $G$ is said to be {\it connected} if every two vertices are connected by a path.
Moreover, since the trivial hypergraph (i.e., $E=\emptyset $) is of  less interest, we only consider  hypergraph having at least one edge (i.e., nontrivial) and  assume that $G$ is simple throughout this paper, which means that $e_{i}\neq e_{j}$ if $i\neq j.$

Now we give the definition of
a {\it strong independent set}~\cite{Balobanov Shabanov 2018} for a
hypergraph.

\begin{definition}[\cite{Balobanov Shabanov 2018}]\label{strong independent set}
{\rm A subset of vertices in a $k$-uniform hypergraph $G$ is called  {\it strong independent}
set if it intersects every edge of $G$ in at most one vertex.}
\end{definition}
It is easy to see that no two vertices of a strong independent set $S$ is adjacent in $G.$
We denote the maximum size of a strong independent set in $G$ as $\alpha_{s}(G).$ A strong independent set with cardinality $\alpha_{s}(G)$ is called a {\it maximum strong independent set.}
Also, a subset of vertices in a $k$-uniform hypergraph $G$ is called  {\it weak independent}
set~\cite{Balobanov Shabanov 2018} if it intersects every edge of $G$ in at most $k-1$ vertices, and
any $k$ vertices of a weak independent set $S$ is not an edge in $G.$
We denote the maximum size of a weak independent set in $G$ as $\alpha(G).$
A weak independent set with cardinality $\alpha(G)$ is called a {\it maximum weak independent set.}
Clearly, for ordinary graphs, i.e. in the case $k=2,$ these notions coincide.

Suppose $G=(V,E)$ is a hypergraph and $f:V \rightarrow \{1,2,\cdots, r\}$ is a vertex coloring with $r$ colors. Then $f$ is said to be {\it strong} for the hypergraph $G,$ if for every edge $e\in E,$ all the vertices in $e$ are colored with distinct colors, i.e.
$$|\{f(v):v\in e\}|=|e|,\qquad \quad \mbox{for all $e \in E$}.$$
Recall that the usual hypergraph chromatic number (i.e. weak chromatic number), $\chi(G),$ introduced by Erd\H{o}s corresponds to weak colorings, i.e. colorings without monochromatic edges when $|\{f(v):v\in e\}|\geq 2$ for  all  $e\in E.$ For ordinary graphs, these notions coincide.

The {\it clique} \cite{Xie Qi 2015} of a $k$-uniform hypergraph $G$ is a set of vertices such that any of its vertex subsets is an edge of $G.$
The largest cardinality of a clique of $G$ is called the {\it clique number} of $G,$ denoted by $\omega(G).$
A clique with cardinality $\omega(G)$ is called a {\it maximum clique.}

A {\it vertex cut}~\cite{Li Cooper Chang 2017} of $G$ is a vertex subset $S\subset V(G)$ such that $G- S$ is disconnected, where $G- S$ is the graph obtained by deleting all vertices in $S$ and all incident edges.
The {\it vertex connectivity} of $G,$ denoted by $\nu(G),$ which is the minimum cardinity of any vertices cut $S.$
A vertex cut  with cardinality $\nu(G)$ is called a {\it minimum vertex cut.}
The complete $k$-uniform hypergraph has no vertex cut. More notations  about hypergraphs  readers are referred to \cite{Berge 1973,Henning Yeo 2020}.

For positive integers $k$ and $n,$ a real {\it tensor} (also called {\it hypermatrix} ) $\mathcal{T}=(t_{i_{1} \ldots i_{k}})$ of order $k$ and dimension
$n$ refers to a multidimensional array with entries $t_{i_{1}\ldots i_{k}}$ such that
$$t_{i_{1}\ldots i_{k}}\in \mathbb{R}, \quad  \mbox {for all} \ i_{j}\in [n]=\{1,2,\ldots ,n\} \ \mbox{and} \ j\in [k].$$
The tensor $\mathcal{T}$ is called {\it symmetric} if  $t_{i_{1}\ldots i_{k}}$ is invariant under any permutation of its indices $i_{1},i_{2},\ldots ,i_{k}.$

A real symmetric tensor $\mathcal{T}$ of order $k$ dimension $n$ uniquely defines a $k$-th degree homogeneous polynomial function with real coefficient by
$$ F_{\mathcal{T}}(x)=\mathcal{T}x^{k}=\sum _{i_{1},\ldots ,i_{k}=1}^{n}t_{i_{1}\ldots i_{k}}x_{i_{1}}\ldots x_{i_{k}}. $$
It is easy to see that $\mathcal{T}x^{k}$ is a real number. Remember that $\mathcal{T}x^{k-1}$ is a vector in $\mathbb{R}^{n},$ whose  $i$-th component is defined as
\begin{equation}\label{TX}
(\mathcal{T}x^{k-1})_{i}=\sum _{i_{2},\ldots,i_{k}=1}^{n}t_{ii_{2}\ldots i_{k}}x_{i_{2}}\ldots x_{i_{k}}.
\end{equation}
\begin{definition}[\cite{Qi2005}]\label{H-eigenvector}
{\rm Let $\mathcal{T}$ be a $k$-th order $n$-dimensional real tensor and $\mathbb{C}$ be the set of all complex numbers.
Then $\lambda$ is an eigenvalue of $\mathcal{T}$ and $0\neq x \in \mathbb{C}^{n}$ is an eigenvector corresponding to $\lambda$
if $(\lambda,x)$ satisfies
$$\mathcal{T}x^{k-1}=\lambda x^{[k-1]},$$
where $x^{[k-1]}\in \mathbb{C}^{n}$ with $(x^{[k-1]})_{i}=(x_{i})^{k-1}.$ }
\end{definition}More information on eigenvalues and eigenvectors of tensors,
the readers are referred to the paper Qi \cite{Qi2005}. Moreover, it is easy to see that
$$(\mathcal{T}x^{k-1})_{i}=
\lambda x_{i}^{k-1}, \ \mbox{for} \ i=1,\ldots,n. $$

Now we introduce the general product \cite{Shao 2013} of tensors, which is a generalization of the matrix case.
\begin{definition}[\cite{Shao 2013}]\label{general tensor product}
{\rm Let $\mathcal{A}$ (and $\mathcal{B}$) be an order $m\geq 2$ (and order $k\geq 1$), dimension $n$ tensor, respectively. Define the product $\mathcal{A}\mathcal{B}$ to be the following tensor $\mathcal{C}$ of order $(m-1)(k-1)+1$ and dimension $n:$
$$c_{i\alpha_{1}\cdots\alpha_{m-1}}=\sum_{i_{2},\cdots, i_{m}=1}^{n}a_{ii_{2}\cdots i_{m}}b_{i_{2}\alpha_{1}}\cdots b_{i_{m}\alpha_{m-1}} \qquad (i\in [n],\alpha_{1},\cdots,\alpha_{m-1}\in [n]^{k-1}).$$}
\end{definition}
Note that by Definition~\ref{general tensor product}, now $\mathcal{T}x^{k-1}$ defined in \eqref{TX} can be
simply written as $\mathcal{T}x.$
The spectral radius of $\mathcal{T}$ is defined as
$\rho(\mathcal{T})=max\{ |\lambda| :\lambda \ \mbox{is an eigenvalue of} \ \mathcal{T} \}.$

For $k\geq 2,$ let $G=(V(G),E(G))$ be a $k$-uniform hypergraph on $n$ vertices.
The {\it adjacency tensor}~\cite{Cooper Dutle}  of $G$ is defined as the $k$-th order $n$-dimensional tensor $\mathcal{A}(G)=(a_{i_{1}\ldots i_{k}})$, where
$$a_{i_{1}\ldots i_{k}}=
\begin{cases}
\frac{1}{(k-1)!}, &  if \text{$ \{i_{1},\ldots,i_{k}\} \in E(G),$}\\
0, & \text{otherwise.}
\end{cases}$$
Let $\mathcal{D}$ be  a $k$-th order $n$-dimensional tensor with its diagonal element $d_{i\ldots i}$ being $d_{i},$ the degree of vertex $i,$ for all $i\in [n],$
then $\mathcal{L}=\mathcal{D}-\mathcal{A}$ is the  {\it Laplacian tensor} of the hypergraph $G,$ and $\mathcal{Q}=\mathcal{D}+\mathcal{A}$
is the {\it signless Laplacian tensor} of the hypergraph $G.$
It is easy to see that both  $\mathcal{A}$ and $\mathcal{Q}$ are  always nonnegative and  symmetric, and  $\mathcal{L}$  is symmetric.

Inspired by the innovating work of Nikiforov \cite{Nikiforov Merging2017}, Lin, Guo and Zhou \cite{Lin Guo Zhou in press} proposed corresponding notation of the convex linear combination $\mathcal{A}_{\alpha}(G)$ of $\mathcal{D}(G)$
and $\mathcal{A}(G),$  which is defined as
$$\mathcal{A}_{\alpha}(G)=\alpha \mathcal{D}(G)+(1-\alpha) \mathcal{A}(G),$$
where $0\leq \alpha <1.$

The spectral radius of $\mathcal{A}_{\alpha}(G)$ is called the {\it $\mathcal{A}_\alpha$-spectral radius} of $G$ and denoted by $\rho_{\alpha}(G).$  Then  $\rho_{0}(G)$ is the  spectral radius of $ \mathcal{A}(G),$  which is called the {\it adjacency spectral radius} of $G,$ and denoted by $\rho(\mathcal{A}(G)).$  Moreover, $2\rho_{1/2}(G)$ is the spectral radius of $\mathcal{Q}(G),$ which is called the {\it signless Laplacian spectral radius} of $G,$ and denoted by $\rho(\mathcal{Q}(G)).$ Also, $\rho(\mathcal{L}(G))$ is the spectral radius of $\mathcal{L}(G),$ which is called the {\it  Laplacian spectral radius} of $G.$

For $k \geq 2,$ let $G$ be a $k$-uniform hypergraph with $V (G) = [n],$ and $x$ be an $n$-dimensional column vector.  Clearly,
$$x^{T}(\mathcal{A}_{\alpha}(G)x)=\alpha\sum_{i\in V(G)}d_{i}x_{i}^{k}+  (1-\alpha)\sum_{e\in E(G)}kx^{e},$$
or equivalently,
$$x^{T}(\mathcal{A}_{\alpha}(G)x)=\sum_{e\in E(G)} \left(\alpha\sum_{i\in e}x_{i}^{k} + (1-\alpha)k x^{e} \right),$$
and $$(\mathcal{A}_{\alpha}(G)x)_{i}=\alpha d_{i}x_{i}^{k-1}+(1-\alpha)\sum_{e\in E_{i}(G)}x^{e \setminus \{i\}},$$
where $x^e=x_{i_{1}}x_{i_{2}}\cdots x_{i_{k}}$ for $e=\{i_{1},i_{2},\cdots,i_{k} \} \in E(G).$

Recently, many researchers focus on the difference between the spectral radius and the average degree of hypergraphs, which is considered as a measure of irregularity.
For $k=2,$ Collatz and Sinogowitz~\cite{Collatz Sinogowitz 1957}  stated that
for any graph with $n$ vertices and $m$ edges,
$\rho_{0}(G)\geq \frac{2m}{n},$
with equality if and only if the graph is regular.
Cioab$\check{a}$ and Gregory~\cite{Cioab Gregory 2007}~ showed that if $G$ is irregular and $n\geq 4,$ then $\rho_{0}(G)- \frac{2m}{n}> \frac{1}{n(\Delta+2)}.$
Ning, Li and Lu~\cite{Ning Li Lu 2013} proved that if $G$ is  irregular and $n\geq 3,$ then $2\rho_{1/2}(G)-\frac{4m}{n}> \frac{(\Delta-\delta)^2}{2n\Delta}.$
For general $k\geq 2,$ Si and Yuan~\cite{Si Yuan 2017} presented that for a $k$-uniform hypergraph $G,$
$\rho_{0}(G)-\frac{2m}{n}\geq \frac{k}{2n}\left(2^{\frac{1}{k}}\left(\Delta^{\frac{k}{k-1}}+\delta^{\frac{k}{k-1}}\right)-(\Delta+\delta) \right),$
which generalized and improved the above result of Cioab$\check{a}$ and Gregory~\cite{Cioab Gregory 2007}.
Besides, for a $k$-uniform hypergraph $G,$ there are many lower bounds on the spectral radius of $\mathcal{A}(G)$ and $\mathcal{Q}(G)$ in terms of  various parameters of hypergraphs, such as  degrees~\cite{Kang Liu Shan 2018, Liu Kang Bai 2019}, co-degrees~\cite{Liu Kang Bai 2019} and  number of edges~\cite{Liu Kang Shan 2018}.

Motivated by the results described above, we consider  $\mathcal{A}_{\alpha}$-spectral analogues for connected $k$-uniform hypergraph.  This paper is organized as follows: In Section 2, we state some basic notations  of tensors and auxiliary lemmas. In Section 3, using the direct product of  tensors with the same order, we present some results of  Laplacian eigenvalue and  $\mathcal{A_{\alpha}}$-spectral radius of the direct product of  hypergraphs, respectively.
In Section 4, we provide several lower bounds on $\mathcal{A}_{\alpha}$-spectral radius of $G$ in terms of vertex degrees.
And also,
we improve a result  in Kang, Liu and Shan~\cite[Theorem~3]{Kang Liu Shan 2018}.
Furthermore, we establish two lower bounds on the $\mathcal{A}_{\alpha}$-spectral radius  using  maximum degree and minimum degree.


\section{Preliminaries}

In this section, we review some notations and helpful lemmas.
For $x\in  \mathbb{R}^{n},$ denote $\|x\|_{k}^{k}=x_{1}^{k}+x_{2}^{k}+\ldots +x_{n}^{k}=\sum _{i=1}^{n}x_{i}^{k}.$ In particular, $x$ is said to be unit if $\|x\|_{k}^{k}=1.$
Denote the set of nonnegative (positive) real vectors of dimension $n$ by $\mathbb{R}_{+}^{n}(\mathbb{R}_{++}^{n}).$

The weak irreducibility of nonnegative tensors was defined in \cite{Friedland Gaubert Han}. It was proved that if $G$ is a connected $k$-uniform  hypergraph with $k\geq 2,$  then $\mathcal{A}_{\alpha}(G)$ is weakly irreducible (see  \cite{Guo Zhou in press}).
The following result is a part of Perron-Frobenius theorem for nonnegative tensors.
\begin{lemma} \label{Perron Frobenius}Let $\mathcal{T}$ be a nonnegative
tensor of order $k$ and dimension $n,$ then we have the following statements.

1. \cite{Yang Yang}   $\rho (\mathcal {T})$ is an eigenvalue of $\mathcal{T}$ with a nonnegative eigenvector corresponding to it.

2. \cite{Friedland Gaubert Han} If furthermore $\mathcal{T}$ is \emph{}symmetric and
weakly irreducible, then $\rho (\mathcal {T})$ is the unique eigenvalue of $\mathcal{T},$ with the unique eigenvector
$x \in \mathbb{R}_{++}^{n},$  up to a positive scaling coefficient.

\end{lemma}

By Lemma~\ref{Perron Frobenius}, for a symmetric weakly irreducible nonnegative tensor $\mathcal{A}_{\alpha}(G),$
$\rho_{\alpha}(G)$ is an eigenvalue of $\mathcal{A}_{\alpha}(G)$ corresponding to a nonnegative eignevector, which is called a {\it Perron vector} of $\mathcal{A}_{\alpha}(G).$ Furthermore, if $G$ is connected, then $\rho_{\alpha}(G)$ is the unique eigenvalue of $\mathcal{A}_{\alpha}(G)$ with the unique eigenvector $x\in \mathbb{R}_{++}^{n},$ up to a positive scaling coefficient.
Thus, if $G$ is a connected $k$-uniform hypergraph, then there is a unique unit positive {\it Perron vector} corresponding to $\rho_{\alpha}(G).$

\begin{lemma}[\cite{Qi2013}]\label{spectral radius}
Let $\mathcal{T}$ be a symmetric nonnegative tensor of order $k$ and dimension $n.$ Then
\begin{equation}\label{the expression of Perron Frobenius}
\rho(\mathcal{T})=max\mathbf{\{}x^{T}(\mathcal{T}x)|x\in \mathbb{R}_{+}^{n},\|x\|_{k}=1 \mathbf{\}}.
\end{equation}
Furthermore, $x \in \mathbb{R}_{+}^{n}$ with $\|x\|_{k}^{k}=1$ is an eigenvector of $\mathcal{T}$ corresponding to
$\rho(\mathcal{T})$ if and only if it is an optimal solution of the above maximization problem \eqref{the expression of Perron Frobenius}.
\end{lemma}

From Lemmas~\ref{Perron Frobenius} and \ref{spectral radius}, for a connected $k$-uniform hypergraph $G$ and a  vector $x\in \mathbb{R}_{+}^{n} $  satisfying $\|x\|_{k}^{k}=1,$ we have $\rho_{\alpha}(G)\geq
x^{T}(\mathcal{A}_{\alpha}(G)x) $ with equlaity if and only if $x$ is the unit positive {\it Perron vector} of $G$.

\begin{definition}[\cite{Shao 2013}]\label{direct product}
{\rm Let $\mathcal{A}$ and $\mathcal{B}$ be two order $k$ tensors with dimension $n$
and $m,$ respectively. Define the direct product $\mathcal{A}\otimes \mathcal{B}$
to be the following tensor of order $k$ and dimension $nm$ (the set of subscripts
is taken as $[n]\times [m]$ in the lexicographic order):

$$(\mathcal{A}\otimes \mathcal{B})_{(i_{1},j_{1})(i_{2},j_{2})\cdots(i_{k}j_{k})}=a_{i_{1}i_{2}\cdots i_{k}}b_{j_{1}j_{2}\cdots j_{k}}.$$ }

\end{definition}
In particular, let $u=(u_{1},u_{2},\cdots,u_{n})^{T}$ and $v=(v_{1},v_{2},\cdots,v_{m})^{T}$ be two column vectors with dimension $n$ and $m,$ respectively. Then
$$u \otimes v=(u_{1}v_{1},u_{2}v_{1},\cdots,u_{n}v_{1},u_{1}v_{2},u_{2}v_{2},\cdots,u_{n}v_{2},\cdots,
u_{1}v_{m},u_{2}v_{m},\cdots,u_{n}v_{m})^{T}.$$

From the above definition, it is easy to have the following proposition.
\begin{proposition}[\cite{Shao 2013}]\label{direct product law1}

$(1)(\mathcal{A}_{1}+\mathcal{A}_{2})\otimes \mathcal{B}=\mathcal{A}_{1}\otimes \mathcal{B}+\mathcal{A}_{2}\otimes \mathcal{B}.$

$(2)\mathcal{A}\otimes (\mathcal{B}_{1}+\mathcal{B}_{2})=\mathcal{A}\otimes \mathcal{B}_{1}+\mathcal{A}\otimes \mathcal{B}_{2}.$

$(3)(\lambda \mathcal{A})\otimes \mathcal{B}=\mathcal{A}\otimes (\lambda \mathcal{B})=\lambda(\mathcal{A}\otimes \mathcal{B}).(\lambda \in \mathbb{C}.)$

\end{proposition}

The following theorem presents an important relation between the direct product of tensors and the general product of tensors  in Definition~\ref{general tensor product}.

\begin{theorem}[\cite{Shao 2013}]\label{direct product law2}
Let $\mathcal{A}$ and $\mathcal{B}$ be two order $k+1$ tensors with dimension $n$ and $m,$ respectively. Let $\mathcal{C}$ and $\mathcal{D}$ be two order $k+1$ tensors with dimension $n$ and $m,$ respectively. Then we have:
$$(\mathcal{A}\otimes\mathcal{B})(\mathcal{C}\otimes\mathcal{D})=(\mathcal{A}\mathcal{C})
\otimes (\mathcal{B}\mathcal{D}).$$
\end{theorem}

In 2013, Shao~\cite{Shao 2013} defined the direct product of two hypergraphs.

\begin{definition}[\cite{Shao 2013}]\label{The direct product of  hypergraphs}(The direct product of  hypergraphs).
{\rm Let $G$ and $H$ be two  $k$-uniform hypergraphs. Define the direct product $G\times H$ of $G$
and $H$ as $V(G\times H)=V(G)\times V(H),$ and $\{(i_{1},j_{1}),\cdots,(i_{k},j_{k})\} \in E(G\times H)$ if and only if $\{i_{1},\cdots,i_{k}\}\in E(G)$ and $\{j_{1},\cdots,j_{k}\}\in E(H).$}
\end{definition}

\begin{lemma}[\cite{Li Cooper Chang 2017}]\label{square inequality}
Let $y_{1},y_{2},\cdots,y_{n}$ be nonnegative numbers ($n\geq 2$). Then
$$\frac{y_{1}+y_{2}+\cdots+y_{n}}{n}-(y_{1}y_{2}\cdots y_{n})^{\frac{1}{n}}\geq \frac{1}{n(n-1)}\sum_{1\leq i< j\leq n}(\sqrt{y_{i}}-\sqrt{y_{j}})^2,$$
equality holds if and only if $y_{1}=y_{2}=\cdots=y_{n}.$
\end{lemma}

\begin{lemma}\label{Jensen's inequality}(Jensen's inequality).
Let $y_{1},y_{2},\cdots,y_{n}$ be real numbers ($n\geq 2$). Then
$$\frac{y_{1}^{n}+y_{2}^{n}+\cdots+y_{n}^{n}}{n}\geq \left(\frac{y_{1}+y_{2}+\cdots+y_{n}}{n}\right)^{n},$$
equality holds if and only if $y_{1}=y_{2}=\cdots=y_{n}.$
\end{lemma}

\section{Some spectral properties of direct product of hypergraphs}

In this section, we  consider some spectral  properties for  Laplacian eigenvalue and  $\mathcal{A_{\alpha}}$-spectral radius of the direct product of  hypergraphs.

According to Definition~\ref{The direct product of  hypergraphs}, for two $k$-uniform hypergraphs $G$ and $H,$  Kang, Liu and Shan \cite[Claim 1]{Kang Liu Shan 2018} proved that when $H=K_{k}^{k},$ (where $K_{k}^{k}$ is the $k$-uniform hypergraph of order $k$ consisting of a single edge), if $G$ is connected, then $\widetilde{G}=G\times K_{k}^{k}$ is connected. Now we generalize this result to the case when $H$ is a connected $k$-uniform hypergraph.

\begin{lemma}\label{A connected}
Let $G$ and $H$ be two $k$-uniform hypergraphs ($k\geq 3$). If both $G$ and $H$ are connected, then $\widetilde{G}:=G\times H$ is connected.
\end{lemma}

\begin{proof}
It suffices to show that for any $(i,s),(j,t)\in V(\widetilde{G}),$ $i,j\in V(G),s,t\in V(H),$ there exists a walk connecting them. We distinguish the following two cases:
\\
$\mathbf{Case~1}.$  There exists an edge $e'\in E(H)$ containing both $s$ and $t$ in $H.$ 

The proof is the same as the proof of  Kang, Liu, and Shan \cite[Claim 1]{Kang Liu Shan 2018}. For the sake of  completeness, we write it down.

$\mathbf{Case~1.1}.$ $i\neq j, s\neq t.$

Since $G$  is connected, there exists a path $P: i=i_{1}e_{1}i_{2}e_{2}\cdots i_{p}e_{p}i_{p+1}=j.$ Since $k\geq 3,$ there exists $s'$ such that
$s'\neq s,s'\neq t.$ From the definition of $\widetilde{G},$ we have the following:

(a) If $p$ is odd, we have

\begin{align*}
\begin{split}
 \left \{
\begin{array}{ll}
    (i_{h},s)\sim (i_{h+1},s'),                    & h=1,3,\cdots,p-2,\\
    (i_{l},s')\sim (i_{l+1},s),                   & l=2,4,\cdots,p-1,\\
    (i_{p},s)\sim (i_{p+1},t)=(j,t).                &
\end{array}
\right.
\end{split}
\end{align*}

(b) If $p$ is even, we have

\begin{align*}
\begin{split}
 \left \{
\begin{array}{ll}
    (i_{h},s)\sim (i_{h+1},s'),                    & h=1,3,\cdots,p-1,\\
    (i_{l},s')\sim (i_{l+1},s),                   & l=2,4,\cdots,p-2,\\
    (i_{p},s')\sim (i_{p+1},t)=(j,t).                &
\end{array}
\right.
\end{split}
\end{align*}
Hence, there exists a walk connecting $(i,s)$ and $(j,t).$

$\mathbf{Case~1.2}.$  $i= j, s\neq t.$

Since $k\geq 3,$ there exist $i'$ and $s'$ satisfying that $i'\neq i,s'\neq s, s'\neq t.$ According to Case~1.1, we know that there exists a path connecting $(i,s)$ and $(i',s').$ Note that $i'\neq j$ and $s'\neq t,$ there is a path connecting $(i',s')$ and $(j,t)$ by Case~1.1. So there exists a walk connecting $(i,s)$ and $(j,t),$ as desired.
\\
$\mathbf{Case~2}.$  There does not exist  an edge containing both $s$ and $t$ in $H.$ 

$\mathbf{Case~2.1}.$  $i\neq j,s\neq t.$

Since $G, H$ are both connected, there exists a path $P:i=i_{1}e_{1}i_{2}e_{2}\cdots i_{p}e_{p}i_{p+1}=j$ connecting $i$ and $j,$ and a path $Q:s=s_{1}f_{1}s_{2}f_{2}\cdots s_{q}f_{q}s_{q+1}=t$ connecting $s$ and $t.$
We distinguish the following two cases:

(a) If $p=q,$ then
$(i,s)=(i_{1},s_{1})\sim (i_{2},s_{2})\sim \cdots \sim(i_{p},s_{p})\sim(i_{p+1},s_{p+1})=(j,t),$ thus there exists a walk connecting $(i,s)$ and $(j,t),$ as desired.

(b) If $p\neq q.$
Without loss of generality, we assume $p>q.$
Since $k\geq 3,$ there exists a vertex $s'\in f_{q}, s'\neq s_{q}, s'\neq t.$
Thus there exists a walk $(i,s)=(i_{1},s_{1})\sim (i_{2},s_{2})\sim \cdots \sim (i_{q},s_{q})\sim (i_{q+1},s')$ connecting $(i,s)$ and $(i_{q+1},s').$

If $p-q$ is odd, we have
\begin{align*}
\begin{split}
 \left \{
\begin{array}{ll}
    (i_{h},s')\sim (i_{h+1},s_{q}),                    & h=q+1,q+3,\cdots,p-2,\\
    (i_{l},s_{q})\sim (i_{l+1},s'),                   & l=q+2,q+4,\cdots,p-1,\\
    (i_{p},s')\sim (i_{p+1},s_{q+1})=(j,t).                &
\end{array}
\right.
\end{split}
\end{align*}
Thus there exists a walk connecting $(i_{q+1},s')$ and $(j,t).$ Hence, there is a walk connecting $(i,s)$ and $(j,t),$ as desired.

If $p-q$ is even, we have
\begin{align*}
\begin{split}
 \left \{
\begin{array}{ll}
    (i_{h},s')\sim (i_{h+1},s_{q}),                    & h=q+1,q+3,\cdots,p-1,\\
    (i_{l},s_{q})\sim (i_{l+1},s'),                   & l=q+2,q+4,\cdots,p-2,\\
    (i_{p},s_{q})\sim (i_{p+1},s_{q+1})=(j,t).                &
\end{array}
\right.
\end{split}
\end{align*}
Thus there exists a walk connecting $(i_{q+1},s')$ and $(j,t).$ Hence, there is a walk connecting $(i,s)$ and $(j,t),$ as desired.

$\mathbf{Case~2.2}.$ $i= j, s\neq t.$

Since $k\geq 3,$ there exist $i'$ and $s'$ satisfying that $i'\neq i,s'\neq s, s'\neq t.$ According to Case~2.1, we know that there exists a path connecting $(i,s)$ and $(i',s').$ Note that $i'\neq j$ and $s'\neq t,$ there is a path connecting $(i',s')$ and $(j,t)$ by Case~2.1. So there exists a walk connecting $(i,s)$ and $(j,t),$ as desired.
\end{proof}

Using the direct product $\widetilde{G}$ of two  connected $k$-uniform hypergraphs $G$ and $H,$ when $H=K_{k}^{k},$ Shao~\cite{Shao 2013} obtained the adjacency spectral radius relationship between $\widetilde{G}:=G\times K_{k}^{k}$ and $G:$ $\rho_{0}(\widetilde{G})=(k-1)!\rho_{0}(G).$  Kang, Liu and Shan~\cite{Kang Liu Shan 2018} obtained the corresponding analogues for signless Lapalcian spectral radius. Here we consider the  corresponding  analogues for  Laplacian eigenvalue,  $\mathcal{A}_{\alpha}$-spectral radius, respectively, and generalize them to the case when $H$ is a connected $d$-regular $k$-uniform hypergraph.

\begin{theorem}\label{direct product spectral radius}
Let $G$ be a  $k$-uniform hypergraph on $n$ vertices with Perron vector $u\in \mathbb{R}^{n}$ corresponding to $\rho(\mathcal{L}(G)).$ Let $\widetilde{G}:=G\times H$ be the direct product of $G$ and $H,$ where $H$ is a  $d$-regular $k$-uniform hypergraph on $m$ vertices. Then
$(k-1)!d\rho(\mathcal{L}(G))$ is an eigenvalue of $\mathcal{L}(\widetilde{G})$ with the corresponding eigenvector $u\otimes e,$
where $e=(1,1,\cdots,1)^{T}\in \mathbb{R}^{m}.$
\end{theorem}

\begin{proof}
Since $H$ is $d$-regular, then for each $i\in V(H),$
$$(\mathcal{L}(H)e)_{i}= d e_{i}^{k-1}-d e_{i}^{k-1}=0,$$
then we have $\mathcal{L}(H)e=0.$
From the definition of the direct product of hypergraphs, it is obvious that $d_{\widetilde{G}}((i,j))=(k-1)!d_{i}d_{j}$ for any $i\in V(G)$ and $j\in V(H).$ Thus  $\mathcal{D}(G\times H)=(k-1)!\mathcal{D}(G)\otimes \mathcal{D}(H),$ and Shao~\cite{Shao 2013} proved that $\mathcal{A}(G\times H)=(k-1)!\mathcal{A}(G)\otimes \mathcal{A}(H).$
From the definition of Laplacian tensor, we have
$$\mathcal{L}(\widetilde{G})=\mathcal{L}(G\times H)= \mathcal{D}(G\times H)-\mathcal{A}(G\times H),$$
or equivalently,
$$\mathcal{L}(\widetilde{G})=(k-1)! \mathcal{D}(G)\otimes \mathcal{D}(H)-(k-1)!\mathcal{A}(G)\otimes\mathcal{A}(H).$$
Therefore, we have
\begin{align*}
\begin{split}
 \mathcal{L}(\widetilde{G})_{(i_{1},j_{1})(i_{2},j_{2})\cdots(i_{k},j_{k})}=\left \{
\begin{array}{ll}
    (k-1)!d_{i_{1}}d_{j_{1}},            & if \quad i_{1}=i_{2}=\cdots=i_{k},\\
    &  \qquad j_{1}=j_{2}=\cdots=j_{k},\\
    -\frac{1}{(k-1)!},                   & if \quad \{i_{1},i_{2},\cdots,i_{k}\}\in E(G),\\
    & \qquad \{j_{1},j_{2},\cdots,j_{k}\}\in E(H),\\
    0,                                   & otherwise.
\end{array}
\right.
\end{split}
\end{align*}
Also, by the definition of direct product of tensors, we have

\begin{align*}
\begin{split}
 \mathcal{L}(G)\otimes \mathcal{L}(H) _{(i_{1},j_{1})(i_{2},j_{2})\cdots(i_{k},j_{k})}=\left \{
\begin{array}{ll}
    d_{i_{1}}d_{j_{1}},            & if \quad i_{1}=i_{2}=\cdots=i_{k},\\
    & \qquad j_{1}=j_{2}=\cdots=j_{k},\\
    \frac{1}{(k-1)!^2},            & if \quad \{i_{1},i_{2},\cdots,i_{k}\}\in E(G),\\
                                  & \qquad \{j_{1},j_{2},\cdots,j_{k}\}\in E(H),\\
    -\frac{1}{(k-1)!}d_{j_{1}},             & if \quad \{i_{1},i_{2},\cdots,i_{k}\}\in E(G),\\
    &  \qquad j_{1}=j_{2}=\cdots=j_{k},\\
    -\frac{1}{(k-1)!}d_{i_{1}},             & if  \quad i_{1}=i_{2}=\cdots=i_{k},\\
    &  \qquad \{j_{1},j_{2},\cdots,j_{k}\}\in E(H),\\
    0,                                  & otherwise.
\end{array}
\right.
\end{split}
\end{align*}
It is easy to see that

$$\mathcal{L}(\widetilde{G})=(k-1)!\left(\mathcal{L}(G)\otimes \mathcal{L}(H)-2\mathcal{A}(G)\otimes\mathcal{A}(H)+\mathcal{A}(G)\otimes \mathcal{D}(H)+\mathcal{D}(G)\otimes \mathcal{A}(H)\right).$$
According to Propositions~\ref{direct product law1} and \ref{direct product law2}, since  $H$ is a $d$-regular $k$-uniform hypergraph, we have
\begin{eqnarray*}
&&\mathcal{L}(\widetilde{G})(u\otimes e)\\
&=&  (k-1)!\left(\mathcal{L}(G)\otimes \mathcal{L}(H)-2\mathcal{A}(G)\otimes\mathcal{A}(H)+\mathcal{A}(G)\otimes \mathcal{D}(H)+\mathcal{D}(G)\otimes \mathcal{A}(H)\right)(u\otimes e)
\\
&=& (k-1)!\big(\mathcal{L}(G)\otimes \mathcal{L}(H)(u\otimes e)-2\mathcal{A}(G)\otimes\mathcal{A}(H)(u\otimes e)+\mathcal{A}(G)\otimes \mathcal{D}(H)(u\otimes e)\\
&&+\mathcal{D}(G)\otimes \mathcal{A}(H)(u\otimes e)\big)
\\
&=& (k-1)!\big((\mathcal{L}(G)u)\otimes (\mathcal{L}(H) e)-2(\mathcal{A}(G)u)\otimes(\mathcal{A}(H)e)+(\mathcal{A}(G)u)\otimes (\mathcal{D}(H) e)\\
&&+(\mathcal{D}(G)u)\otimes (\mathcal{A}(H) e)\big)
\\
&=& (k-1)!\left((-\mathcal{A}(G)u)\otimes(\mathcal{A}(H)e)+(\mathcal{D}(G)u)\otimes (\mathcal{A}(H) e)\right)
\\
&=& (k-1)!(\mathcal{L}(G)u)\otimes(\mathcal{A}(H)e)
\\
&=& (k-1)!\left(\mathcal{L}(G)u\otimes de\right)
\\
&=& (k-1)!d\rho(\mathcal{L}(G))\left(u\otimes e\right)
.
\end{eqnarray*}
Thus $(k-1)!d\rho(\mathcal{L}(G))$ is an eigenvalue of $\mathcal{L}(\widetilde{G})$ with the corresponding eigenvector $u\otimes e.$ Hence the proof is completed.
\end{proof}

\begin{theorem}\label{direct product spectral radius}
Let $G$ be a connected $k$-uniform hypergraph on $n$ vertices with Perron vector $u\in \mathbb{R}_{+}^{n}$ corresponding to $\rho_{\alpha}(G)$ $(k\geq3).$ Let $\widetilde{G}:=G\times H$ be the product of $G$ and $H,$ where $H$ is a connected $d$-regular $k$-uniform hypergraph on $m$ vertices. Then $\rho_{\alpha}(\widetilde{G})=(k-1)!d\rho_{\alpha}(G)$ and $u\otimes e$ is an eigenvector corresponding to $\rho_{\alpha}(\widetilde{G}),$ where $e=(1,1,\cdots,1)^{T}\in \mathbb{R}^{m}.$
\end{theorem}

\begin{proof}
Since $H$ is $d$-regular,  then for each $i\in V(H),$
$$(\mathcal{A}_{\alpha}(H)e)_{i}=\alpha d e_{i}^{k-1}+(1-\alpha)d e_{i}^{k-1}=de_{i}^{k-1},$$
then we have $\mathcal{A}_{\alpha}(H)e=de,$ by the connectedness of $H$ and   Lemma~\ref{Perron Frobenius}, we have $\rho_{\alpha}(H)=d.$
Since $\mathcal{D}(G\times H)=(k-1)!\mathcal{D}(G)\otimes \mathcal{D}(H),$ and Shao~\cite{Shao 2013} proved that $\mathcal{A}(G\times H)=(k-1)!\mathcal{A}(G)\otimes \mathcal{A}(H).$
From the definition of $\mathcal{A}_{\alpha}$-tensor, we have
$$\mathcal{A}_{\alpha}(\widetilde{G})=\mathcal{A}_\alpha(G\times H)=\alpha \mathcal{D}(G\times H)+(1-\alpha)\mathcal{A}(G\times H),$$
or equivalently,
$$\mathcal{A}_{\alpha}(\widetilde{G})=\alpha(k-1)! \mathcal{D}(G)\otimes \mathcal{D}(H)+(1-\alpha)(k-1)!\mathcal{A}(G)\otimes\mathcal{A}(H).$$
Since $\mathcal{A}_{\alpha}(G)=\alpha\mathcal{D}(G)+(1-\alpha)\mathcal{A}(G),$
where $0\leq \alpha <1,$ we have
$$\mathcal{A}(G)=\frac{\mathcal{A}_{\alpha}(G)-\alpha\mathcal{D}(G)}{1-\alpha}.$$
According to Proposition~\ref{direct product law1} and Proposition~\ref{direct product law2}, since  $H$ is  $d$-regular, we have
 \begin{eqnarray*}
&&\mathcal{A}_\alpha(\widetilde{G})(u\otimes e)\\
&=&  \left(\alpha(k-1)! \mathcal{D}(G)\otimes \mathcal{D}(H)+(1-\alpha)(k-1)!\mathcal{A}(G)\otimes\mathcal{A}(H)\right)(u\otimes e)
\\
&=& (k-1)!\left(\alpha \mathcal{D}(G)\otimes \mathcal{D}(H)+(1-\alpha)\mathcal{A}(G)\otimes\mathcal{A}(H)\right)(u\otimes e)
\\
&=& (k-1)!\left(\alpha \mathcal{D}(G)\otimes\mathcal{D}(H)+(1-\alpha)\frac{\mathcal{A}_{\alpha}(G)-\alpha\mathcal{D}(G)}{1-\alpha}\otimes\mathcal{A}(H)\right)(u\otimes e)
\\
&=& (k-1)!\left(\alpha \mathcal{D}(G)\otimes \mathcal{D}(H)+\left(\mathcal{A}_{\alpha}(G)-\alpha\mathcal{D}(G)\right)\otimes\mathcal{A}(H)\right)(u\otimes e)
\\
&=& (k-1)!\big(\alpha \mathcal{D}(G)\otimes \mathcal{D}(H)(u\otimes e)
\\
&&+\mathcal{A}_{\alpha}(G)\otimes\mathcal{A}(H)(u\otimes e)-\alpha\mathcal{D}(G)\otimes\mathcal{A}(H)(u\otimes e)\big)
\\
&=& (k-1)!\left(\alpha \mathcal{D}(G)u\otimes \mathcal{D}(H)e+\mathcal{A}_{\alpha}(G)u\otimes\mathcal{A}(H)e-\alpha\mathcal{D}(G)u\otimes\mathcal{A}(H) e\right)
\\
&=& (k-1)!\left(\alpha \mathcal{D}(G)u\otimes de+\mathcal{A}_{\alpha}(G)u\otimes\mathcal{A}(H)e-\alpha\mathcal{D}(G)u\otimes de\right)
\\
&=& (k-1)!\left(\mathcal{A}_{\alpha}(G)u\otimes\mathcal{A}(H)e\right)
\\
&=& (k-1)!\left(\rho_{\alpha}(G)u\otimes de\right)
\\
&=& (k-1)!d\rho_{\alpha}(G)\left(u\otimes e\right)
.
\end{eqnarray*}\emph{}
Since $G, H$ are both connected, then $\widetilde{G}$ is connected by Lemma \ref{A connected}. Obviously, we have $\mathcal{A}_{\alpha}(\widetilde{G})$ is weakly irreducible. So by Lemma~\ref{Perron Frobenius}, we have $(k-1)!d\rho_{\alpha}(G)$ is the spectral radius of $\mathcal{A}_{\alpha}(\widetilde{G})$ with the corresponding eigenvector $u\otimes e.$ Hence the proof is completed.
\end{proof}



\section{Lower bounds for the $\mathcal{A}_{\alpha}$-Spectral radius of uniform hypergraphs}
In this section,  we present some lower bounds for the $\mathcal{A}_{\alpha}$-spectral radius of $k$-uniform hypergraphs.

Let $G$ be a $k$-uniform hypergraph on $n$ vertices with $m$ edges. By Lemma~\ref{spectral radius}, it is easy to see that if $x$ is a unit column vector in $\mathbb{R}^{n},$ then
$$\rho_{\alpha}(G)\geq x^{T}(\mathcal{A}_{\alpha}(G)x)= \alpha \sum_{i=1}^{n}d_{i}x_{i}^{k}+(1-\alpha)\sum_{e\in E(G)}kx^{e}.$$
Moreover, if $G$ is connected, equality holds if and only if $x$ is an eigenvector corresponding to $\rho_{\alpha}(G).$ By Lemma~\ref{spectral radius} and choosing $x=(\frac{1}{\sqrt[k]{n}},\frac{1}{\sqrt[k]{n}},\cdots,\frac{1}{\sqrt[k]{n}})^{T}$ in the above inequality, we  have the following result.

\begin{lemma}\label{average degree}
Let $G$ be a 
$k$-uniform hypergraph on $n$ vertices and $m$ edges, then
$$\rho_{\alpha}(G)\geq \frac{km}{n}.$$
If $G$ is connected, equality holds if and only if $G$ is regular.
\end{lemma}

We now obtain a lower bound on $\rho_{\alpha}(G)- \frac{km}{n}$ in terms of the restrictions $d_{S}=\left(d_{i}\right)_{i\in S}$ of the degree sequence
of $G$ to a subset $S$ of the vertex set.

\begin{lemma}\label{main lemma1}
Let $G$ be a $k$-uniform hypergraph on $n$ vertices with $m$ edges, and $S$ be a strong independent set of G with $|S|=s.$ Then
$$\rho_{\alpha}(G)- \frac{km}{n}\geq \frac{1}{n}\left(\alpha\left(s\frac{\sum_{i\in S}d_{i}^{\frac{2k-1}{k-1}}}{\sum_{i\in S}d_{i}^{\frac{k}{k-1}}}-\sum_{i\in S}d_{i}\right)+(1-\alpha)k\left(s^{\frac{1}{k}}\left(\sum_{i\in S}d_{i}^{\frac{k}{k-1}}\right)^{\frac{k-1}{k}}-\sum_{i\in S}d_{i}\right)\right).$$
\end{lemma}

\begin{proof}
Let $s=|S|$ denote the number of vertices in $S.$ Taking $x_{i}=\frac{a_{i}}{\sqrt[k]{n}}$ for $i \in S$ and $x_{i}=\frac{1}{\sqrt[k]{n}}$
when $i\notin S.$ We have $\parallel x\parallel_{k}=1$ when $\sum_{i\in S}a_{i}^{k}=s.$ Since  $S$ is a strong independent set of $G,$
we have $|e\cap S|\leq1$ for each edge $e\in E(G).$ Then we have

\begin{eqnarray*}
&&\rho_{\alpha}(G)-\frac{km}{n}
\\
&\geq &  x^T (\mathcal{A}_{\alpha}(G)x)-\frac{km}{n}
\\
&=& \alpha \sum_{i\in V(G)}d_{i}x_{i}^{k}+(1-\alpha)k \sum_{e\in E(G)} x^{e}-\frac{km}{n}
\\
&=& \alpha \left(\sum_{i\in S}d_{i}x_{i}^{k}+\sum_{i\in V\setminus S}d_{i}x_{i}^{k}\right)+ (1-\alpha)k \left(\sum_{e\in E(G), e\cap S\neq \emptyset} x^{e}+\sum_{e\in E(G), e\cap S= \emptyset} x^{e}\right)-\frac{km}{n}
\\
&=& \alpha \left(\sum_{i\in S}d_{i}x_{i}^{k}+\sum_{i\in V\setminus S}d_{i}x_{i}^{k}\right)
\\
&&+(1-\alpha)k \left(\sum_{\substack{\{i_{1},i_{2},\cdots,i_{k}\}\in E(G)\\ \{i_{1},i_{2},\cdots,i_{k}\}\cap S\neq \emptyset}} x_{i_{1}}x_{i_{2}}\cdots x_{i_{k}}
+\sum_{\substack{\{i_{1},i_{2},\cdots,i_{k}\}\in E(G)\\ \{i_{1},i_{2},\cdots,i_{k}\}\cap S= \emptyset}} x_{i_{1}}x_{i_{2}}\cdots x_{i_{k}}\right)-\frac{km}{n}
\\
&=& \alpha \left(\sum_{i\in S}d_{i}\frac{a_{i}^{k}}{n}+\sum_{i\in V\setminus S}d_{i}\frac{1}{n}\right)+
   (1-\alpha)k \left(\sum_{i\in S} d_{i}\frac{a_{i}}{n}+\sum_{\substack{e \in E(G)\\ e\cap S= \emptyset}} \frac{1}{n}\right)-\frac{km}{n}
\\
&=& \frac{\alpha}{n} \left(\sum_{i\in S}d_{i}a_{i}^{k}+\left(km-\sum_{i\in S}d_{i}\right)\right)+
   \frac{(1-\alpha)k}{n} \left(\sum_{i\in S} d_{i}a_{i}+\left(m-\sum_{i\in S}d_{i}\right) \right)-\frac{km}{n}
\\
&=& \frac{\alpha}{n} \left(\sum_{i\in S}d_{i}a_{i}^{k}-\sum_{i\in S}d_{i}\right)+
   \frac{(1-\alpha)k}{n} \left(\sum_{i\in S} d_{i}a_{i}-\sum_{i\in S}d_{i} \right)
\end{eqnarray*}
If we choose the $a_{i}$ so that the equality holds in the H\"{o}lder inequality, $\sum_{i\in S} d_{i}a_{i}\leq \left(\sum_{i\in S}a_{i}^{k}\right)^{\frac{1}{k}}\left(\sum_{i\in S}d_{i}^{\frac{k}{k-1}}\right)^{\frac{k-1}{k}},$
equality holds if and only if $a_{i}^{k}=c^{k}d_{i}^{\frac{k}{k-1}},$ i.e., $a_{i}=cd_{i}^{\frac{1}{k-1}},$  where $c$ is a constant. Therefore, we have $c=\sqrt[k]{\frac{s}{\sum_{i\in S}d_{i}^{\frac{k}{k-1}}}}$ by $s=\sum_{i\in S}a_{i}^{k}=c^{k}\sum_{i\in S}d_{i}^{\frac{k}{k-1}}.$
Thus we have
\begin{eqnarray*}
&&\rho_{\alpha}(G)-\frac{km}{n}
\\
&\geq & \frac{\alpha}{n} \left(\sum_{i\in S}d_{i}a_{i}^{k}-\sum_{i\in S}d_{i}\right)+
   \frac{(1-\alpha)k}{n} \left(\sum_{i\in S} d_{i}a_{i}-\sum_{i\in S}d_{i} \right)
\\
&=& \frac{\alpha}{n} \left(\sum_{i\in S}d_{i}\frac{sd_{i}^{\frac{k}{k-1}}}{\sum_{i\in S}d_{i}^{\frac{k}{k-1}}}-\sum_{i\in S}d_{i}\right)+
   \frac{(1-\alpha)k}{n} \left(\sum_{i\in S} d_{i}\sqrt[k]{\frac{sd_{i}^{\frac{k}{k-1}}}{\sum_{i\in S}d_{i}^{\frac{k}{k-1}}}}-\sum_{i\in S}d_{i} \right)
\\
&=&\frac{1}{n}\left(\alpha\left(s\frac{\sum_{i\in S}d_{i}^{\frac{2k-1}{k-1}}}{\sum_{i\in S}d_{i}^{\frac{k}{k-1}}}-\sum_{i\in S}d_{i}\right)+(1-\alpha)k\left(s^{\frac{1}{k}}\left(\sum_{i\in S}d_{i}^{\frac{k}{k-1}}\right)^{\frac{k-1}{k}}-\sum_{i\in S}d_{i}\right)\right).
\end{eqnarray*}
We complete the proof.
\end{proof}

Taking $S$  be a maximum strong independent set in Lemma~\ref{main lemma1}, we obtain the following corollary.
\begin{corollary}
Let $G$ be a  $k$-uniform hypergraph on $n$ vertices with $m$ edges, and $S$ be a maximum strong independent set of $G,$ then
\begin{eqnarray*}
&&\rho_{\alpha}(G)- \frac{km}{n}\\
&\geq & \frac{\alpha}{n}\left(\alpha_{s}(G)\frac{\sum_{i\in S}d_{i}^{\frac{2k-1}{k-1}}}{\sum_{i\in S}d_{i}^{\frac{k}{k-1}}}-\sum_{i\in S}d_{i}\right)+\frac{(1-\alpha)k}{n}\left(\alpha_{s}(G)^{\frac{1}{k}}\left(\sum_{i\in S}d_{i}^{\frac{k}{k-1}}\right)^{\frac{k-1}{k}}-\sum_{i\in S}d_{i}\right).
\end{eqnarray*}
\end{corollary}

\begin{theorem}\label{main theorem1}
Let $G$ be a connected $k$-uniform hypergraph on $n$ vertices with $m$ edges $(k\geq 3),$ and $S$ be a subset of $V(G)$ with $|S|=s.$ Then
\begin{eqnarray*}
&&\rho_{\alpha}(G)- \frac{km}{n}\\
&\geq& \frac{1}{kn}\left(\alpha\left(s\frac{\sum_{i\in S}d_{i}^{\frac{2k-1}{k-1}}}{\sum_{i\in S}d_{i}^{\frac{k}{k-1}}}-\sum_{i\in S}d_{i}\right)+(1-\alpha)k\left(s^{\frac{1}{k}}\left(\sum_{i\in S}d_{i}^{\frac{k}{k-1}}\right)^{\frac{k-1}{k}}-\sum_{i\in S}d_{i}\right)\right).
\end{eqnarray*}
\end{theorem}

\begin{proof}
If  $S$ is a strong independent set, then by Lemma~\ref{main lemma1}, we obtain the result.
If $S$ is not a strong independent set, then we consider the direct product $\widetilde{G}:=G\times K_{k}^{k}.$ Clearly,
$\widetilde{G}$ is a $k$-partite $k$-uniform hypergraph with partition:
$$V(\widetilde{G})=\bigcup\limits_{j=1}^{k}(V(G)\times \{j\}).$$
Clearly, $|V(\widetilde{G})|=k|V(G)|,$ $|E(\widetilde{G})|=k!|E(G)|,$ and $d_{\widetilde{G}}((i,a))=(k-1)!d_{i},$
where $d_{i}$ is the degree of vertex $i$ in $G,$ $i\in V(G),a\in [k]=V(K_{k}^{k}).$

Note that $S\times \{a\}$ is a strong independent set in $\widetilde{G}.$ Applying  the previous inequality, we have
\begin{eqnarray*}
&&\rho_{\alpha}(\widetilde{G})- \frac{kk!m}{kn}
\\
&\geq &  \frac{(k-1)!}{kn}\left(\alpha\left(s\frac{\sum_{i\in S}d_{i}^{\frac{2k-1}{k-1}}}{\sum_{i\in S}d_{i}^{\frac{k}{k-1}}}-\sum_{i\in S}d_{i}\right)+(1-\alpha)k\left(s^{\frac{1}{k}}\left(\sum_{i\in S}d_{i}^{\frac{k}{k-1}}\right)^{\frac{k-1}{k}}-\sum_{i\in S}d_{i}\right)\right).
\end{eqnarray*}

Since $K_{k}^{k}$ is  $1$-regular and $G$ is connected,   by Theorem~\ref{direct product spectral radius}, we know that $\rho_{\alpha}(\widetilde{G})=(k-1)!\rho_{\alpha}(G),$ hence we have
\begin{eqnarray*}
&&\rho_{\alpha}(G)- \frac{km}{n}\\
&\geq& \frac{1}{kn}\left(\alpha\left(s\frac{\sum_{i\in S}d_{i}^{\frac{2k-1}{k-1}}}{\sum_{i\in S}d_{i}^{\frac{k}{k-1}}}-\sum_{i\in S}d_{i}\right)+(1-\alpha)k\left(s^{\frac{1}{k}}\left(\sum_{i\in S}d_{i}^{\frac{k}{k-1}}\right)^{\frac{k-1}{k}}-\sum_{i\in S}d_{i}\right)\right).
\end{eqnarray*}
The proof is completed.
\end{proof}

\begin{remark}\label{inequality}
{\rm For a subset $S$ of the vertex set $V(G)$ of a connected $k$-uniform hypergraph $G$ $(k\geq 3),$ on the one hand, it follows from the Rearrangement  inequality that
$$s\sum_{i\in S}d_{i}^{\frac{2k-1}{k-1}}=s\sum_{i\in S}d_{i}^{\frac{k}{k-1}}d_{i}\geq \sum_{i\in S}d_{i}^{\frac{k}{k-1}}\sum_{i\in S}d_{i},$$
equality holds if and only if $d_{i}$ is a constant  for any $i\in S.$
On the other hand, when $p=k,q=\frac{k}{k-1},$ it follows from the H\"{o}lder inequality that
$$s^{\frac{1}{k}}\left(\sum_{i\in S} d_{i}^{\frac{k}{k-1}}\right)^{\frac{k-1}{k}}\geq \sum_{i\in S}d_{i},$$
equality holds if and only if $d_{i}$ is a constant for any $i\in S.$
Therefore, by $0\leq \alpha <1,$ we have
\begin{eqnarray*}
&&\rho_{\alpha}(G)
\\
&\geq & \frac{1}{kn}\left(\alpha\left(s\frac{\sum_{i\in S}d_{i}^{\frac{2k-1}{k-1}}}{\sum_{i\in S}d_{i}^{\frac{k}{k-1}}}-\sum_{i\in S}d_{i}\right)+(1-\alpha)k\left(s^{\frac{1}{k}}\left(\sum_{i\in S}d_{i}^{\frac{k}{k-1}}\right)^{\frac{k-1}{k}}-\sum_{i\in S}d_{i}\right)\right)+\frac{km}{n}
\\
&\geq & \frac{km}{n},
\end {eqnarray*}
which improves the corresponding result in  Lemma~\ref{average degree}.} 
\end{remark}


Taking $S=V(G),$ we obtain the following corollary of Theorem~\ref{main theorem1}.
\begin{corollary}
Let $G$ be a connected $k$-uniform hypergraph on $n$ vertices with $m$ edges $(k\geq 3).$ Then
\begin{eqnarray}
\rho_{\alpha}(G)
\geq \frac{\alpha}{k}\frac{\sum_{i=1}^{n}d_{i}^{\frac{2k-1}{k-1}}}{\sum_{i=1}^{n}d_{i}^{\frac{k}{k-1}}}+(1 -\alpha)\left(\frac{1}{n}\sum_{i=1}^{n}d_{i}^{\frac{k}{k-1}}\right)^{\frac{k-1}{k}}+ \frac{\alpha (k-1)m}{n}.
\end{eqnarray}
\end{corollary}

\begin{remark}\label{Liu's result}
{\rm Let $G$ be a $k$-uniform hypergraph on $n$ vertices ($k\geq 3$) and $\widetilde{G}:=G\times K_{k}^{k},$
Kang, Liu and Shan in \cite[Claim~4]{Kang Liu Shan 2018} proved that $\rho_{0}(\widetilde{G})=(k-1)!\rho_{0}(G).$
According to the proof of Theorem~\ref{main theorem1} and Lemma~\ref{Perron Frobenius},
we have
$$\rho({\mathcal{A}}(G))=\rho_{0}(G)\geq \left(\frac{1}{n}\sum_{i=1}^{n}d_{i}^{\frac{k}{k-1}}\right)^{\frac{k-1}{k}}.$$
If $G$ is connected, equality holds if and only if $G$ is regular, which was proved in \cite[Theorem~2]{Kang Liu Shan 2018}.}
\end{remark}

Taking $S$  be a pair of vertices with distinct degrees in Theorem~\ref{main theorem1}, we obtain the following corollary. 
\begin{corollary}\label{two distinct vertices}
Let $G$ be a connected $k$-uniform hypergraph on $n$ vertices with $m$ edges $(k\geq 3)$, suppose $i$ and $j$ are vertices of $G$ and $d_{i}>d_{j}.$ Then
\begin{eqnarray*}
&&\rho_{\alpha}(G)- \frac{km}{n}
\\
&\geq& \frac{\alpha}{cn}\left(2\frac{d_{i}^{\frac{2k-1}{k-1}}+d_{j}^{\frac{2k-1}{k-1}}}{d_{i}^{\frac{k}{k-1}}+d_{j}^{\frac{k}{k-1}}}- (d_{i}+d_{j})\right)+\frac{(1-\alpha)k}{cn}\left(2^{\frac{1}{k}}\left(d_{i}^{\frac{k}{k-1}}+d_{j}^{\frac{k}{k-1}}\right)^{\frac{k-1}{k}}-(d_{i}+d_{j})\right)
\\
&>& \frac{\alpha}{cn}\left(\frac{d_{i}^{\frac{2k-1}{k-1}}+d_{j}^{\frac{2k-1}{k-1}}}{d_{i}^{\frac{k}{k-1}}}- (d_{i}+d_{j})\right)+\frac{(1-\alpha)k}{cn}\left(2^{\frac{1}{k}}\left(d_{i}^{\frac{k}{k-1}}+d_{j}^{\frac{k}{k-1}}\right)^{\frac{k-1}{k}}-(d_{i}+d_{j})\right),
\end{eqnarray*}
where $c=1$ if $i$ and $j$ are not adjacent and $c=k$ if $i$ and $j$ are adjacent.
\end{corollary}

The next result is an immediate consequence of Corollary~\ref{two distinct vertices}.
\begin{corollary}\label{first maximum and minimum degree}
Let $G$ be a connected $k$-uniform hypergraph on $n$ vertices with $m$ edges $(k\geq 3)$, suppose $d_{i}=\Delta, d_{j}=\delta,
 i\neq j$, and $i,j\in V(G).$ Then
\begin{eqnarray}
&&\rho_{\alpha}(G)- \frac{km}{n} \nonumber
\\
&\geq& \frac{\alpha}{cn}\left(2\frac{\Delta^{\frac{2k-1}{k-1}}+\delta^{\frac{2k-1}{k-1}}}{\Delta^{\frac{k}{k-1}}+\delta^{\frac{k}{k-1}}}- (\Delta+\delta)\right)+\frac{(1-\alpha)k}{cn}\left(2^{\frac{1}{k}}\left(\Delta^{\frac{k}{k-1}}+\delta^{\frac{k}{k-1}}\right)^{\frac{k-1}{k}}-(\Delta+\delta)\right)
\\
&\geq& \frac{\alpha}{cn}\left(\frac{\Delta^{\frac{2k-1}{k-1}}+\delta^{\frac{2k-1}{k-1}}}{\Delta^{\frac{k}{k-1}}}- (\Delta+\delta)\right)+\frac{(1-\alpha)k}{cn}\left(2^{\frac{1}{k}}\left(\Delta^{\frac{k}{k-1}}+\delta^{\frac{k}{k-1}}\right)^{\frac{k-1}{k}}-(\Delta+\delta)\right) \nonumber
\end{eqnarray}
where $c=1$ if $i$ and $j$ are not adjacent and $c=k$ if $i$ and $j$ are adjacent.
\end{corollary}

Taking $S$ be a maximum weak independent set  in Theorem~\ref{main theorem1}, we obtain the following corollary.
\begin{corollary}\label{weak independence number}
Let $G$ be a connected $k$-uniform hypergraph on $|V(G)|=n$ vertices with $|E(G)|=m$ edges $(k\geq 3)$, and $S$ be a maximum weak independent set of $G.$ Then
\begin{eqnarray*}
&&\rho_{\alpha}(G)- \frac{km}{n}\\
&\geq & \frac{\alpha}{kn}\left(\alpha(G)\frac{\sum_{i\in S}d_{i}^{\frac{2k-1}{k-1}}}{\sum_{i\in S}d_{i}^{\frac{k}{k-1}}}-\sum_{i\in S}d_{i}\right)+\frac{(1-\alpha)}{n}\left(\alpha(G)^{\frac{1}{k}}\left(\sum_{i\in S}d_{i}^{\frac{k}{k-1}}\right)^{\frac{k-1}{k}}-\sum_{i\in S}d_{i}\right).
\end{eqnarray*}
\end{corollary}

Berge \cite{Berge 1973} proved that  every $k$-uniform hypergraph $G$ on $n$ vertices satisfying that $\chi(G)\alpha(G)\geq n.$
According to Corollary~\ref{weak independence number}, we have the following result.

\begin{corollary}
Let $G$ be a connected $k$-uniform hypergraph on $n$ vertices with $m$ edges $(k\geq 3)$, $\chi(G) $ is the weak chromatic number of $G,$ and $S$ be a maximum  weak independent set.  Then
\begin{eqnarray*}
&&\rho_{\alpha}(G)- \frac{km}{n}\\
&\geq & \frac{\alpha}{kn}\left(\frac{n}{\chi(G)}\frac{\sum_{i\in S}d_{i}^{\frac{2k-1}{k-1}}}{\sum_{i\in S}d_{i}^{\frac{k}{k-1}}}-\sum_{i\in S}d_{i}\right)+\frac{(1-\alpha)}{n}\left(\left(\frac{n}{\chi(G)}\right)^{\frac{1}{k}}\left(\sum_{i\in S}d_{i}^{\frac{k}{k-1}}\right)^{\frac{k-1}{k}}-\sum_{i\in S}d_{i}\right).
\end{eqnarray*}
\end{corollary}

For a $k$-uniform hypergraph $G,$  notice that a maximum weak  independence set of $G$ is a  maximum clique of $\widetilde{G},$ and  there is a relation \cite{Xie Qi 2015} between clique number $\omega(\overline{G})$ and weak independence number $\alpha(G):$ $\alpha(G)=\omega(\overline{G}).$

\begin{corollary}
Let $G$ be a connected $k$-uniform hypergraph on $n$ vertices with $m$ edges $(k\geq 3)$,
$\omega(\overline{G})$ is the clique number of $\overline{G},$
and $S$ be a maximum clique of $\overline{G}.$
Then
\begin{eqnarray*}
&&\rho_{\alpha}(G)- \frac{km}{n}\\
&\geq & \frac{\alpha}{kn}\left(\omega(\overline{G})\frac{\sum_{i\in S}d_{i}^{\frac{2k-1}{k-1}}}{\sum_{i\in S}d_{i}^{\frac{k}{k-1}}}-\sum_{i\in S}d_{i}\right)+\frac{(1-\alpha)}{n}\left(\omega(\overline{G})^{\frac{1}{k}}\left(\sum_{i\in S}d_{i}^{\frac{k}{k-1}}\right)^{\frac{k-1}{k}}-\sum_{i\in S}d_{i}\right).
\end{eqnarray*}
\end{corollary}



Taking $S$  be a minimum  vertex cut  in Theorem~\ref{main theorem1}, we obtain the following corollary.
\begin{corollary}
Let $G$ be a connected $k$-uniform hypergraph on $n$ vertices with $m$ edges $(k\geq 3)$, and $S$ be a minimum vertex cut of $G.$ Then
\begin{eqnarray*}
&&\rho_{\alpha}(G)- \frac{km}{n}\\
&\geq& \frac{\alpha}{cn}\left(\nu(G)\frac{\sum_{i\in S}d_{i}^{\frac{2k-1}{k-1}}}{\sum_{i\in S}d_{i}^{\frac{k}{k-1}}}-\sum_{i\in S}d_{i}\right)+\frac{(1-\alpha)k}{cn}\left(\nu(G)^{\frac{1}{k}}\left(\sum_{i\in S}d_{i}^{\frac{k}{k-1}}\right)^{\frac{k-1}{k}}-\sum_{i\in S}d_{i}\right),
\end{eqnarray*}
where $c=1$ if $S$ is a strong independent set of $G,$ and $c=k$ otherwise.
\end{corollary}


\begin{lemma}\label{square lemma}
Let $G$ be a 
$k$-uniform hypergraph on $n$ vertices with $m$ edges, and $S$ be a strong independent set of $G$ with $|S|=s.$ Then
$$\rho_{\alpha}(G)- \frac{km}{n}\geq \frac{1}{n}\left(\alpha\sum_{i\in S}d_{i}\left(\left(\frac{sd_{i}^{\frac{k}{k-1}}}{\sum_{i\in S}d_{i}^{\frac{k}{k-1}}}\right)^{\frac{1}{2}}-1\right)^{2}+k\left(s^{\frac{1}{k}}\left(\sum_{i\in S}d_{i}^{\frac{k}{k-1}}\right)^{\frac{k-1}{k}}-\sum_{i\in S}d_{i}\right)\right).$$
\end{lemma}

\begin{proof}
Let $s=|S|$ denote the number of vertices in $S.$ Taking $x_{i}=\frac{a_{i}}{\sqrt[k]{n}}$ for $i \in S$ and $x_{i}=\frac{1}{\sqrt[k]{n}}$
when $i\notin S.$ We have $\parallel x\parallel_{k}=1$ when $\sum_{i\in S}a_{i}^{k}=s.$ Since  $S$ is a strong independent set, thus we have $|e\cap S|\leq 1$ for each edge $e\in E(G).$ Then we have
\begin{eqnarray*}
&&\rho_{\alpha}(G)-\frac{km}{n}
\\
&\geq & \alpha \sum_{i\in V(G)}d_{i}x_{i}^{k}+(1-\alpha)k \sum_{e\in E(G)} x^{e}-\frac{km}{n}
\\
&=& \alpha \sum_{\{i_{1},i_{2},\cdots,i_{k}\}\in E(G)}\left(x_{i_{1}}^{k}+x_{i_{2}}^{k}+\cdots+x_{i_{k}}^{k}\right)+ (1-\alpha)k \sum_{\{i_{1},i_{2},\cdots,i_{k}\}\in E(G)}x_{i_{1}}x_{i_{2}}\cdots x_{i_{k}}-\frac{km}{n}
\\
&=& \alpha \left(\sum_{\{i_{1},i_{2},\cdots,i_{k}\}\in E(G)}\left(x_{i_{1}}^{k}+x_{i_{2}}^{k}+\cdots+x_{i_{k}}^{k}\right)-k \sum_{\{i_{1},i_{2},\cdots,i_{k}\}\in E(G)}x_{i_{1}}x_{i_{2}}\cdots x_{i_{k}}\right)
\\
&&+ k \sum_{\{i_{1},i_{2},\cdots,i_{k}\}\in E(G)}x_{i_{1}}x_{i_{2}}\cdots x_{i_{k}}-\frac{km}{n}.
\end{eqnarray*}
Using  Lemma~\ref{square inequality}, we have
\begin{eqnarray*}
&&\rho_{\alpha}(G)-\frac{km}{n}
\\
&\geq & \alpha k \left(\sum_{\{i_{1},i_{2},\cdots,i_{k}\}\in E(G)}\left(\frac{x_{i_{1}}^{k}+x_{i_{2}}^{k}+\cdots+x_{i_{k}}^{k}}{k}-x_{i_{1}}x_{i_{2}}\cdots x_{i_{k}}\right)\right)
\\
&&+ k \sum_{\{i_{1},i_{2},\cdots,i_{k}\}\in E(G)}x_{i_{1}}x_{i_{2}}\cdots x_{i_{k}}-\frac{km}{n}
\\
&\geq &\alpha k\frac{1}{k(k-1)}\sum_{e\in E(G)}\sum_{\{i,j\}\subseteq e}\left(x_{i}^{\frac{k}{2}}-x_{j}^{\frac{k}{2}}\right)^{2}+ k \sum_{\{i_{1},i_{2},\cdots,i_{k}\}\in E(G)}x_{i_{1}}x_{i_{2}}\cdots x_{i_{k}}-\frac{km}{n}
\\
&= &\alpha \frac{1}{k-1}\sum_{e\in E(G)}\sum_{\{i,j\}\subseteq e}\left(x_{i}^{\frac{k}{2}}-x_{j}^{\frac{k}{2}}\right)^{2}+ k \sum_{\{i_{1},i_{2},\cdots,i_{k}\}\in E(G)}x_{i_{1}}x_{i_{2}}\cdots x_{i_{k}}-\frac{km}{n}
\\
&\geq &\frac{1}{n}\left(\alpha\sum_{i\in S}d_{i}\left(a_{i}^{\frac{k}{2}}-1\right)^{2}+k\left(\sum_{i\in S} d_{i}a_{i}-\sum_{i\in S}d_{i} \right)\right).
\end{eqnarray*}
Then by the similar arguments in Lemma~\ref{main lemma1}, we obtain the desired result.
\end{proof}

\begin{theorem}\label{main theorem2}
Let $G$ be a connected $k$-uniform hypergraph on $n$ vertices with $m$ edges $(k\geq 3),$ and $S$ be a subset of $V(G)$ with $|S|=s.$ Then
\begin{eqnarray*}
&&\rho_{\alpha}(G)- \frac{km}{n}
\\
&\geq& \frac{1}{kn}\left(\alpha\sum_{i\in S}d_{i}\left(\left(\frac{sd_{i}^{\frac{k}{k-1}}}{\sum_{i\in S}d_{i}^{\frac{k}{k-1}}}\right)^{\frac{1}{2}}-1\right)^{2}+k\left(s^{\frac{1}{k}}\left(\sum_{i\in S}d_{i}^{\frac{k}{k-1}}\right)^{\frac{k-1}{k}}-\sum_{i\in S}d_{i}\right)\right).
\end{eqnarray*}
\end{theorem}

\begin{proof}
By  Theorem~\ref{direct product spectral radius} and Lemma~\ref{square lemma}, using the similar method in Theorem~\ref{main theorem1}, we obtain the desired result.
\end{proof}

Taking $S=V(G)$ in Theorem~\ref{main theorem2}, we obtain the following  result.
\begin{corollary}\label{signless Laplacian bound2}
Let $G$ be a connected $k$-uniform hypergraph on $n$ vertices with $m$ edges $(k\geq 3).$  Then
\begin{eqnarray}
\rho_{\alpha}(G)
\geq  \frac{\alpha}{kn}\sum_{i=1}^{n}d_{i}\left(\left(\frac{nd_{i}^{\frac{k}{k-1}}}{\sum_{i=1}^{n}d_{i}^{\frac{k}{k-1}}}\right)^{\frac{1}{2}}-1\right)^{2}+\left(\frac{1}{n}\sum_{i=1}^{n}d_{i}^{\frac{k}{k-1}}\right)^{\frac{k-1}{k}}.
\end{eqnarray}
\end{corollary}

\begin{remark}
{\rm Let $G$ be a connected $k$-uniform hypergraph on $|V(G)|=n$ vertices $(k\geq 3),$
taking $\alpha=\frac{1}{2}$ in Corollary~\ref{signless Laplacian bound2}. Then we have
\begin{eqnarray}
\rho(\mathcal{Q}(G))=2\rho_{\frac{1}{2}}(G)
&&\geq \frac{1}{kn}\sum_{i=1}^{n} d_{i}\left(\left(\frac{nd_{i}^{\frac{k}{k-1}}}{\sum_{i=1}^{n}d_{i}^{\frac{k}{k-1}}}\right)^{\frac{1}{2}}-1\right)^2+2\left(\frac{1}{n}\sum_{i=1}^{n} d_{i}^{\frac{k}{k-1}} \right)^{\frac{k-1}{k}}\nonumber \\
&&\geq 2\left(\frac{1}{n}\sum_{i=1}^{n} d_{i}^{\frac{k}{k-1}} \right)^{\frac{k-1}{k}},
\end{eqnarray}
which improves the result in  Kang, Liu and Shan~\cite[Theorem~3]{Kang Liu Shan 2018}.}
\end{remark}

\begin{remark}

{\rm Nikiforov \cite{Nikiforov 2017} introduced the concept of odd-colorable hypergraphs, which is a generalization of bipartite graphs. Let $k\geq 2$ and $k$ be even. A $k$-uniform hypergraph $G$ with $V(G)=[n]$ is called {\it odd-colorable} if there exists a map $\varphi:[n]\rightarrow [k]$ such that for any edge
$\{i_{1},i_{2},\cdots,i_{k}\}$ of $G,$ we have
$$\varphi(i_{1})+\varphi(i_{2})+\cdots+\varphi(i_{k})\equiv \frac{k}{2}(mod~k).$$
It was proved that if $G$ is a connected $k$-uniform hypergraph, then $\rho(\mathcal{L}(G))=\rho(\mathcal{Q}(G))$ if and only if $k$ is even and $G$
is odd-colorable \cite{Yuan Qi Shao Ouyang 2018}.
Thus by  $\rho_{\frac{1}{2}}(G)=\frac{1}{2}\rho(\mathcal{Q}(G)),$ we have the following results.
Let $k\geq 4,$ and $k$ be even,  for a connected odd-colorable $k$-uniform hypergraph $G,$ we have
\begin{eqnarray*}
\rho(\mathcal{L}(G))=2\rho_{\frac{1}{2}}(G)
&&\geq \frac{1}{kn}\sum_{i=1}^{n} d_{i}\left(\left(\frac{nd_{i}^{\frac{k}{k-1}}}{\sum_{i=1}^{n}d_{i}^{\frac{k}{k-1}}}\right)^{\frac{1}{2}}-1\right)^2+2\left(\frac{1}{n}\sum_{i=1}^{n} d_{i}^{\frac{k}{k-1}} \right)^{\frac{k-1}{k}}\\
&&\geq 2\left(\frac{1}{n}\sum_{i=1}^{n} d_{i}^{\frac{k}{k-1}} \right)^{\frac{k-1}{k}},\\
\end{eqnarray*}
which improves the result in Kang, Liu and Shan~\cite[Remark~12]{Kang Liu Shan 2018}.}
\end{remark}

\begin{lemma}\label{k square lemma}
Let $G$ be a 
$k$-uniform hypergraph on $n$ vertices with $m$ edges $(k\geq 3),$ and $S$ be a strong independent set of $G.$  Then
\begin{eqnarray*}
&&\rho_{\alpha}(G)- \frac{km}{n} \\
&\geq& \frac{\alpha}{nk^{k-1}}\sum_{i\in S}d_{i}\left(\left(\sqrt[k]{\frac{sd_{i}^{\frac{k}{k-1}}}{\sum_{i\in S}d_{i}^{\frac{k}{k-1}}}}+k-1\right)^{k}-k^{k}\right)\\
&&+\frac{(1-\alpha)k}{n}\left(s^{\frac{1}{k}}\left(\sum_{i\in S}d_{i}^{\frac{k}{k-1}}\right)^{\frac{k-1}{k}}-\sum_{i\in S}d_{i}\right).
\end{eqnarray*}
\end{lemma}

\begin{proof}
Let $s=|S|$ denote the number of vertices in $S.$ Taking $x_{i}=\frac{a_{i}}{\sqrt[k]{n}}$ for $i \in S$ and $x_{i}=\frac{1}{\sqrt[k]{n}}$
when $i\notin S.$ We have $\parallel x\parallel_{k}=1$ when $\sum_{i\in S}a_{i}^{k}=s.$ Since  $S$ is a strong independent set, thus we have $|e\cap S|\leq 1$ for each edge $e\in E(G).$ Then we have
\begin{eqnarray*}
&&\rho_{\alpha}(G)-\frac{km}{n}
\\
&=& \alpha \sum_{i\in V(G)}d_{i}x_{i}^{k}+(1-\alpha)k \sum_{e\in E(G)} x^{e}-\frac{km}{n}
\\
&=& \alpha \sum_{\{i_{1},i_{2},\cdots,i_{k}\}\in E(G)}(x_{i_{1}}^{k}+x_{i_{2}}^{k}+\cdots+x_{i_{k}}^{k})+ (1-\alpha)k \sum_{\{i_{1},i_{2},\cdots,i_{k}\}\in E(G)}x_{i_{1}}x_{i_{2}}\cdots x_{i_{k}}-\frac{km}{n}.
\end{eqnarray*}
By Jensen's inequality, 
we have
\begin{eqnarray*}
&&\rho_{\alpha}(G)-\frac{km}{n}
\\
&\geq& \alpha \sum_{\{i_{1},i_{2},\cdots,i_{k}\}\in E(G)}(x_{i_{1}}^{k}+x_{i_{2}}^{k}+\cdots+x_{i_{k}}^{k})+ (1-\alpha)k \sum_{\{i_{1},i_{2},\cdots,i_{k}\}\in E(G)}x_{i_{1}}x_{i_{2}}\cdots x_{i_{k}}-\frac{km}{n}
\\
&\geq& \alpha \sum_{\{i_{1},i_{2},\cdots,i_{k}\}\in E(G)}\frac{(x_{i_{1}}+x_{i_{2}}+\cdots+x_{i_{k}})^{k}}{k^{k-1}}+ (1-\alpha)k \sum_{\{i_{1},i_{2},\cdots,i_{k}\}\in E(G)}x_{i_{1}}x_{i_{2}}\cdots x_{i_{k}}-\frac{km}{n}
\\
&= & \frac{\alpha}{k^{k-1}} \sum_{\{i_{1},i_{2},\cdots,i_{k}\}\in E(G)}(x_{i_{1}}+x_{i_{2}}+\cdots+x_{i_{k}})^{k}\\
&&+ (1-\alpha)k \sum_{\{i_{1},i_{2},\cdots,i_{k}\}\in E(G)}x_{i_{1}}x_{i_{2}}\cdots x_{i_{k}}-\frac{km}{n}
\\
&= & \frac{\alpha}{nk^{k-1}} \sum_{i \in S}d_{i}\left((a_{i}+k-1)^{k}-k^{k}\right)+ \frac{(1-\alpha)k}{n} \left(\sum_{i \in S} a_{i}d_{i}-\sum_{i \in S}d_{i}\right).
\end{eqnarray*}
Again by the similar arguments in Lemma~\ref{main lemma1}, we obtain the desired result.
\end{proof}

\begin{theorem}\label{main theorem3}
Let $G$ be a connected $k$-uniform hypergraph on $n$ vertices with $m$ edges $(k\geq 3),$ and $S$ be a  subset with $|S|=s.$ Then
\begin{eqnarray*}
&&\rho_{\alpha}(G)- \frac{km}{n} \\
&\geq& \frac{\alpha}{k^{k}n}\sum_{i\in S}d_{i}\left(\left(\sqrt[k]{\frac{sd_{i}^{\frac{k}{k-1}}}{\sum_{i\in S}d_{i}^{\frac{k}{k-1}}}}+k-1\right)^{k}-k^{k}\right)
+\frac{(1-\alpha)}{n}\left(s^{\frac{1}{k}}\left(\sum_{i\in S}d_{i}^{\frac{k}{k-1}}\right)^{\frac{k-1}{k}}-\sum_{i\in S}d_{i}\right).
\end{eqnarray*}

\end{theorem}

\begin{proof}
By  Theorem~\ref{direct product spectral radius} and Lemma~\ref{k square lemma}, using the similar method in Theorem~\ref{main theorem1}, we obtain the desired result.
\end{proof}

\begin{remark}
{\rm For a connected $k$-uniform hypergraph $G$ on $n$ vertices with $m$ edges ($k\geq 3$), let $S$ be a subset of $G,$ we have
$$\sum_{i\in S}d_{i}\left(\sqrt[k]{\frac{sd_{i}^{\frac{k}{k-1}}}{\sum_{i\in S}d_{i}^{\frac{k}{k-1}}}}+k-1\right)^{k}\geq \sum_{i\in S}d_{i}\sum_{i\in S}\left(\sqrt[k]{\frac{sd_{i}^{\frac{k}{k-1}}}{\sum_{i\in S}d_{i}^{\frac{k}{k-1}}}}+k-1\right)^{k}\geq k^{k}\sum_{i\in S}d_{i},$$
the first inequality follows from the Rearrangement inequality,
At the same time, considering
$\sum_{i\in S}\left(\sqrt[k]{\frac{sd_{i}^{\frac{k}{k-1}}}{\sum_{i\in S}d_{i}^{\frac{k}{k-1}}}}+k-1\right)^{k},$
we notice that there exists an $ i_{0}\in S$ satisfying that $\sqrt[k]{\frac{sd_{i_{0}}^{\frac{k}{k-1}}}{\sum_{i\in S}d_{i}^{\frac{k}{k-1}}}}\geq 1,$
thus the second inequality holds.

Furthermore,  it follows from the H\"{o}lder inequality, we have
$$s^{\frac{1}{k}}\left(\sum_{i\in S}d_{i}^{\frac{k}{k-1}}\right)^{\frac{k-1}{k}}\geq \sum_{i\in S}d_{i},$$
equality holds if and only if $d_{i}$ is a constant  for any $i\in S.$
Thus we have
\begin{eqnarray*}
&&\rho_{\alpha}(G) \\
&\geq& \frac{\alpha}{k^{k}n}\sum_{i\in S}d_{i}\left(\left(\sqrt[k]{\frac{sd_{i}^{\frac{k}{k-1}}}{\sum_{i\in S}d_{i}^{\frac{k}{k-1}}}}+k-1\right)^{k}-k^{k}\right)\\
&&+\frac{(1-\alpha)}{n}\left(s^{\frac{1}{k}}\left(\sum_{i\in S}d_{i}^{\frac{k}{k-1}}\right)^{\frac{k-1}{k}}-\sum_{i\in S}d_{i}\right)+\frac{km}{n}
\\
&\geq & \frac{km}{n},
\end{eqnarray*}
which improves the corresponding result in Lemma~\ref{average degree}, and generalizes the result in \cite[Theorem~3.8]{Cooper Dutle}.}
\end{remark}

Taking $S=V(G)$ in Theorem~\ref{main theorem3}, we obtain the following  result.
\begin{corollary}\label{signless Laplacian bound3}
Let $G$ be a connected $k$-uniform hypergraph on $n$ vertices with $m$ edges $(k\geq 3).$  Then
\begin{eqnarray}
\rho_{\alpha}(G)
\geq \frac{\alpha}{k^{k}n}\sum_{i=1}^{n}d_{i}\left(\sqrt[k]{\frac{nd_{i}^{\frac{k}{k-1}}}{\sum_{i=1}^{n}d_{i}^{\frac{k}{k-1}}}}+k-1\right)^{k}
+(1-\alpha)\left(\frac{1}{n}\sum_{i=1}^{n} d_{i}^{\frac{k}{k-1}}\right)^{\frac{k-1}{k}}.
\end{eqnarray}
\end{corollary}

\begin{remark}
{\rm Let $G$ be a connected  $k$-uniform hypergraph on $|V(G)|=n$ vertices and $|E(G)|=m$ edges ($k\geq 3$), and $S$ be a subset of $V(G),$
notice that the lower bound of $\rho_{\alpha}(G)-\frac{km}{n}$ in Theorem~\ref{main theorem1} is better than the  corresponding lower bound in  Theorems~\ref{main theorem2} and ~\ref{main theorem3}.}
\end{remark}

\begin{remark}
{\rm Let $G$ be a $k$-uniform hypergraph on $n$ vertices and $m$ edges ($k\geq 3$). Notice that taking $S$ be  a {\it transversal}~\cite{Henning Yeo 2020} (also called {\it vertex cover } or {\it hitting set} ), or a {\it dominating set}~\cite{Bujtas Henning Tuza 2012} in Theorems~\ref{main theorem1}, \ref{main theorem2} and \ref{main theorem3}, we deduce the corresponding results, respectively,  since
$S$ is an arbitrary subset of vertex set  $V(G)$ in these three above theorems.
Moreover, taking $S$ be a maximum strong independent set, a maximum weak independent set, a minimum vertex cut in Theorems~\ref{main theorem2} and \ref{main theorem3}, we also deduce the corresponding results, respectively.}
\end{remark}


\begin{theorem}\label{second maximum minimum theorem}
Suppose that $G$ is a $k$-uniform hypergraph with $n$ vertices and $m$ edges. Let $\Delta$ and $\delta$ be the maximum and
minimum degrees of $G,$ respectively. Then
\begin{eqnarray}
&&\rho_{\alpha}(G)-\frac{km}{n}\nonumber\\
&\geq& \frac{\alpha}{2n}\left(\frac{2\left(\Delta^{\frac{2k-1}{k-1}}+
\delta^{\frac{2k-1}{k-1}}\right)}{\Delta^{\frac{k}{k-1}}+
\delta^{\frac{k}{k-1}}}-(\Delta+\delta)\right)+\frac{(1-\alpha)k}{2n}\left(2^{\frac{1}{k}}\left(\Delta^{\frac{k}{k-1}}+\delta^{\frac{k}{k-1}}\right)^{\frac{k-1}{k}}-(\Delta+\delta
)\right).\nonumber
\end{eqnarray}

\end{theorem}
\begin{proof}
Let $i_{0}$ and $j_{0}$ be the vertices of $G$ with $d_{i_{0}}=\Delta$  and $d_{j_{0}}=\delta.$ We distinguish the following two cases.\\
$\mathbf{Case~1}.$ $i_{0}$ and $j_{0}$ are not adjacent in $G.$

We first define a vector $x\in \mathbb{R}^{n}$ by
\begin{align*}
\begin{split}
 x_{i}=\left \{
\begin{array}{ll}
    \frac{a_{1}}{\sqrt[k]{n}},                    & i=i_{0},\\
    \frac{a_{2}}{\sqrt[k]{n}},                   & i=j_{0},\\
    \frac{1}{\sqrt[k]{n}},                & otherwise,
\end{array}
\right.
\end{split}
\end{align*}
where
$$a_{1}=\frac{\sqrt[k]{2}\Delta^{\frac{1}{k-1}}}{\sqrt[k]{\Delta^{\frac{k}{k-1}}+\delta^{\frac{k}{k-1}}}},
a_{2}=\frac{\sqrt[k]{2}\delta^{\frac{1}{k-1}}}{\sqrt[k]{\Delta^{\frac{k}{k-1}}+\delta^{\frac{k}{k-1}}}}.$$
It can be checked that $\|x\|_{k}=1.$ By Lemma~\ref{spectral radius}, we have
\begin{eqnarray*}
&&\rho_{\alpha}(G)-\frac{km}{n}
\\
&\geq &  x^T (\mathcal{A}_{\alpha}(G)x)-\frac{km}{n}
\\
&=& \alpha \sum_{i\in V(G)}d_{i}x_{i}^{k}+(1-\alpha)k \sum_{e\in E(G)} x^{e}-\frac{km}{n}
\\
&=& \frac{\alpha}{n}\left(\frac{2\left(\Delta^{\frac{2k-1}{k-1}}+
\delta^{\frac{2k-1}{k-1}}\right)}{\Delta^{\frac{k}{k-1}}+
\delta^{\frac{k}{k-1}}}-(\Delta+\delta)+km\right)\\
&&+\frac{(1-\alpha)k}{n}\left(2^{\frac{1}{k}}\left(\Delta^{\frac{k}{k-1}}+\delta^{\frac{k}{k-1}}\right)^{\frac{k-1}{k}}-(\Delta+\delta)+m\right)-\frac{km}{n}
\\
&\geq&\frac{\alpha}{n}\left(\frac{2\left(\Delta^{\frac{2k-1}{k-1}}+
\delta^{\frac{2k-1}{k-1}}\right)}{\Delta^{\frac{k}{k-1}}+
\delta^{\frac{k}{k-1}}}-(\Delta+\delta)\right)+\frac{(1-\alpha)k}{n}\left(2^{\frac{1}{k}}\left(\Delta^{\frac{k}{k-1}}+\delta^{\frac{k}{k-1}}\right)^{\frac{k-1}{k}}-(\Delta+\delta
)\right).
\end{eqnarray*}
$\mathbf{Case~2}.$ $i_{0}$ and $j_{0}$ are adjacent in $G.$

Let $G'$ be a copy of $G$ and $G^{*}=G\cup G'.$ Let $i_{0}'$
and $j_{0}'$ be the corresponding vertices of $i_{0}$ and $j_{0}$  in $G',$ respectively. Clearly, $i_{0}$ and $j_{0}'$
are not adjacent in $G^{*},$ and $d_{G^{*}}(i_{0})=\Delta,$ $d_{G^{*}}(j_{0}')=\delta.$ Using the same arguments of Case~1 for $G^{*},$ we have
\begin{eqnarray*}
&&\rho_{\alpha}(G)-\frac{km}{n}
\\
&=&\rho_{\alpha}(G^{*})-\frac{km}{n}
\\
&\geq& \frac{\alpha}{2n}\left(\frac{2\left(\Delta^{\frac{2k-1}{k-1}}+
\delta^{\frac{2k-1}{k-1}}\right)}{\Delta^{\frac{k}{k-1}}+
\delta^{\frac{k}{k-1}}}-(\Delta+\delta)\right)+\frac{(1-\alpha)k}{2n}\left(2^{\frac{1}{k}}\left(\Delta^{\frac{k}{k-1}}+\delta^{\frac{k}{k-1}}\right)^{\frac{k-1}{k}}-(\Delta+\delta
)\right).
\end{eqnarray*}
We complete the proof.
\end{proof}

\begin{remark}
{\rm Let $G$ be a $k$-uniform hypergraph on $n$ vertices and $m$ edges.
Similar to  Remark~\ref{inequality}, it follows from the  Rearrangement inequality and the H\"{o}lder inequality that
\begin{eqnarray*}
&&\rho_{\alpha}(G)-\frac{km}{n}
\\
&\geq & \frac{\alpha}{2n}\left(\frac{2\left(\Delta^{\frac{2k-1}{k-1}}+
\delta^{\frac{2k-1}{k-1}}\right)}{\Delta^{\frac{k}{k-1}}+
\delta^{\frac{k}{k-1}}}-(\Delta+\delta)\right)+\frac{(1-\alpha)k}{2n}\left(2^{\frac{1}{k}}\left(\Delta^{\frac{k}{k-1}}+\delta^{\frac{k}{k-1}}\right)^{\frac{k-1}{k}}-(\Delta+\delta
)\right)
\\
&\geq& 0,
\end{eqnarray*}
which improves the result in Lemma~\ref{average degree} and generalizes the result of Si and Yuan \cite[Theorem~3.2]{Si Yuan 2017}.

Furthermore, for a connected $k$-uniform hypergraph $G$ on $n$ vertices and $m$ edges ($k\geq 3$),
taking $S$  be a pair of vertices $\{i,j\}$ with $d_{i}=\Delta, d_{j}=\delta,
 i\neq j$, and $i,j\in V(G).$ When $i,j$ are not adjacent in $G,$  the first lower bound of $\rho_{\alpha}(G)$ in Corollary~\ref{first maximum and minimum degree} is better than the lower bound in Theorem~\ref{second maximum minimum theorem}. When $i,j$ are adjacent in $G,$  the lower bound of $\rho_{\alpha}(G)$ in Theorem~\ref{second maximum minimum theorem} is better than the first lower bound in  Corollary~\ref{first maximum and minimum degree}.}

\end{remark}






\end{document}